\documentclass[a4paper,11pt]{article}
\usepackage[utf8]{inputenc}
\usepackage[margin=3cm]{geometry}
\usepackage[numbers]{natbib}
\usepackage{authblk}
\usepackage{charter}
\usepackage{hyperref}
\usepackage{color,colortbl}
\usepackage{longtable}
\usepackage{amsthm,amsmath,amssymb,mathtools}
\usepackage[capitalize]{cleveref}
\usepackage{todonotes}
\usepackage{nicefrac}

\newtheorem{thm}{Theorem}

\newtheorem{lemma}[thm]{Lemma}
\theoremstyle{definition}
\newtheorem{definition}{Definition}

\title{Improved lower bound on the dimension of the EU council's voting rules}
\author[1]{Stefan Kober}
\author[2]{Stefan Weltge}
\affil[1]{\textit{\small{Advanced Optimization in a Networked Economy (AdONE), Technical University of Munich, Arcisstr. 21, 80333 Munich
}}}
\affil[2]{\textit{\small{Technical University of Munich, Arcisstr. 21, 80333 Munich }}}
\date{}

\newcommand{\winning}{\mathcal{W}}
\newcommand{\winningeu}{\mathcal{W}_{\text{EU}}}
\newcommand{\losing}{\mathcal{L}}
\newcommand{\nonsep}{\mathcal{N}}
\newcommand{\calF}{\mathcal{F}}
\newcommand{\R}{\mathbb{R}}

\begin{document}

\maketitle
\begin{abstract}
	Kurz and Napel (2015) proved that the voting system of the EU council (based on the 2014 population data) cannot be represented as the intersection of six weighted games, i.e., its dimension is at least $ 7 $.
	This set a new record for real-world voting rules and the authors posed the exact determination as a challenge.
	Recently, Chen, Cheung, and Ng (2019) showed that the dimension is at most $ 24 $.

	We provide the first improved lower bound and show that the dimension is at least $ 8 $.\\
	
\end{abstract}

\section{Introduction}\label{sec:intro}

Simple games are cooperative games that are commonly used to describe real-world voting systems.
Considering a fixed, finite set $ M $ of voting members, a \emph{simple game} is given by a collection $ \winning $ of subsets of $ M $ satisfying the monotonicity property: $ C \in \winning $ and $ C \subseteq C' \subseteq M $ implies $ C' \in \winning $.
The sets in $ \winning $ are called \emph{winning coalitions}, and each subset of $ M $ that is not in $ \winning $ is called a \emph{losing coalition}.
A fundamental class of simple games are \emph{weighted games} whose winning coalitions can be written as
\[
	\winning = \left \{ C \subseteq M : \sum \nolimits_{m \in C} a_m \ge \beta \right \}
\]
for some $ a \in \R_{\ge 0}^M $ and $ \beta \in \R $.
It is a basic fact that every simple game is the intersection of finitely many weighted games, and hence we may define the \emph{dimension} of a simple game $ \winning $ to be the smallest number of weighted games whose intersection is $ \winning $.

Determining the dimension of (simple games associated to) real-world voting systems has been of particular interest in social choice theory, see, e.g., \citet*{TaylorPacelli}.
While many voting rules are actually weighted and hence have dimension one, examples of dimension two are given by the US federal legislative system~\cite{TaylorZwicker} and the amendment of the Canadian constitution~\cite{Kilgour}. A voting rule of dimension three has been adopted by the Legislative Council of Hong Kong~\cite{CheungNg}.

A new record was set with the change of the EU (European Union) council's voting system by the Treaty of Lisbon in 2014.
Based on the population data of 2014, \citet*{Kurz_2015} showed that its dimension is at least $ 7 $ and at most $13368$, and they posed the exact determination as a challenge to the community.
In response, \citet*{Chen_2019} were able to reduce the upper bound to $24$.

We provide the first improved lower bound and show that the dimension is at least $ 8 $.
Although we will not rely on this interpretation in what follows, the idea behind our lower bound is based on the observation that the dimension of a simple game $ \winning $ can be seen as the chromatic number of a particular hypergraph $ H $: the nodes of $ H $ are the losing coalitions, and a set of losing coalitions $ \nonsep $ forms a hyperedge iff $ \nonsep \cap \winning' \ne \emptyset $ for every weighted game $ \winning' \supseteq \winning $.
The proof of \citet{Kurz_2015} establishes that $ H $ contains a \emph{clique} of cardinality $ 7 $, which directly implies that the chromatic number of $ H $ is at least $ 7 $.
This idea has been used previously in the context of lower bounds on sizes of integer programming formulations~\cite{Jeroslow,KaibelWeltge,FaenzaSanita}.
While we have not found any \emph{simple} subgraph of larger chromatic number, we will show that $ H $ contains a \emph{hyper}graph on $ 15 $ nodes whose chromatic number is $ 8 $.

\paragraph{Outline.}
In Section~\ref{secStrategy} we introduce the concept of non-separable subsets of the losing coalitions of a simple game $ \winning $.
A family $ \calF $ of such subsets can be thought of as a subgraph of the above hypergraph.
Moreover, we consider the notion of a $ k $-cover for such a set $ \calF $, which can be seen as a node-coloring of the respective subgraph with $ k $ colors.
Accordingly, we will see that if the dimension of $ \winning $ is at most $ k $, then there exists a $ k $-cover for each $ \calF $.
In Section~\ref{secConstruction} we consider the simple game associated to the EU council and give a construction of a set $ \calF $, for which no $ 7 $-cover exists.
A proof of the latter fact will be given in Section~\ref{secLowerBound}.

\section{Strategy}
\label{secStrategy}
In what follows, we consider simple games on a common fixed ground set $ M $.

\begin{definition}\label{def:non-separable}
	Let $ \winning $ be a simple game and $ \nonsep $ be any set of losing coalitions of $ \winning $.
	We say that $ \nonsep $ is \emph{non-separable} with respect to $ \winning $ if every weighted game $ \winning' \supseteq \winning $ satisfies $ \winning' \cap \nonsep \ne \emptyset $.
\end{definition}

From the definition it is immediate that a simple game is weighted if and only if no set of losing coalitions is non-separable.
So, the existence of a single non-separable set yields that the dimension of a simple game is at least two.
To obtain a larger lower bound, the following notion will be useful.

\begin{definition}
	Let $ \winning $ be a simple game with losing coalitions $ \losing $, and let $ \nonsep_1, \dots, \nonsep_t \subseteq \losing $ be non-separable with respect to $ \winning $.
	A \emph{$ k $-cover} of $ (\nonsep_1, \dots, \nonsep_t) $ is a collection of sets $ \losing_1,\dots,\losing_k \subseteq \losing $ such that
	\begin{enumerate}
		\item\label{enum:union} $ \losing_1 \cup \dots \cup \losing_k = \nonsep_1 \cup \dots \cup \nonsep_t $ and
		\item\label{enum:contain} $ \nonsep_i \nsubseteq \losing_j $ for all $ i \in \{1,\dots,t\} $, $ j \in \{1,\dots,k\} $.
	\end{enumerate}
	
\end{definition}

In order to obtain a lower bound on the dimension, we will exploit the following observation.

\begin{lemma}\label{lemma:dim}
	Let $ \winning $ be a simple game with non-separable sets $ \nonsep_1, \dots, \nonsep_t $.
	If $ \winning $ has dimension at most $ k $, then there exists a $ k $-cover for $ (\nonsep_1, \dots, \nonsep_t) $.
\end{lemma}
\begin{proof}
	If $ \winning $ has dimension at most $ k $, then there exist $k$ weighted games $ \winning_1, \dots, \winning_k $ such that $ \bigcap_{i=1}^k \winning_i = \winning $.
	For $ i \in \{1,\dots,k\} $ define $ \losing_i $ as the intersection of the losing coalitions in $ \winning_i $ and $ \losing^* \coloneqq \nonsep_1 \cup \dots \cup \nonsep_t $.

	We claim that $(\losing_1, \dots, \losing_k)$ is a $ k $-cover of $ (\nonsep_1, \dots, \nonsep_t) $.
	In order to show Property~\ref{enum:union}, first observe that $ \losing_1 \cup \dots \cup \losing_k \subseteq \losing^* $ holds.
	Now, for any $ \ell \in \losing^* \subseteq \losing $ we have $ \ell \notin \winning $ and hence there is an $ i \in \{1,\dots,k\} $ with $ \ell \notin \winning_i $, which implies $ \ell \in \losing_i $.

	For Property~\ref{enum:contain}, assume that $ \nonsep_i \subseteq \losing_j $ holds for some $ i \in \{1,\dots,t\} $ and $ j \in \{1,\dots,k\} $.
	This means that each coalition in $ \nonsep_i $ is losing for $ \winning_j $, meaning that $ \nonsep_i $ and $ \winning_j $ are disjoint.
	This contradicts the fact that $ \nonsep_i $ is non-separable with respect to $ \winning $ since $ \winning_j \supseteq \winning $ is weighted.
\end{proof}

In what follows, we will consider the simple game associated with the EU council and construct a collection of non-separable losing coalitions that does not permit a $ 7 $-covering.
By \cref{lemma:dim} this implies that the dimension must be at least $ 8 $.

\section{Our Construction}
\label{secConstruction}
Let us give a formal definition of the simple game associated to the EU council based on the population data of 2014, as considered by~\citet{Kurz_2015}.
In 2014, the European Union consisted of $ 28 $ members and hence we may fix $ M \coloneqq \{1,\dots,28\} $.
In the voting system of the EU council, a coalition is winning if
\begin{enumerate}
	\item \label{enum:55} it contains at least $ 55\% $ of all members states \emph{and}
	\item \label{enum:65} it unites at least $ 65\% $ of the total EU population,
\end{enumerate}
or
\begin{enumerate}
	\setcounter{enumi}{2}
	\item \label{enum:25} it consists of at least $25$ of the $28$ member states.
\end{enumerate}
Denoting the weighted game associated with rule $i$ by $\winning_i$ and the simple game that represents the voting system of the EU council by $\winningeu$, we thus have
\[
	\winningeu = (\winning_{\ref{enum:55}} \cap \winning_{\ref{enum:65}}) \cup \winning_{\ref{enum:25}}.
\]
Note that $ \winning_{\ref{enum:65}} $ depends on the population of each member state.
As in~\cite{Kurz_2015}, we will work with the data depicted in Table~\ref{tabpopulation}.
From these numbers, it can be seen that the following coalitions are losing with respect to $ \winningeu $.
\begingroup
\allowdisplaybreaks
\begin{align*}
	L_{1}  & \coloneqq \{ 2, 3, 5, 6, 8, 9, 10, 11, 12, 15, 16, 17, 18, 19, 20, 21, 22, 23, 24, 25, 26, 27, 28 \} \\
	L_{2}  & \coloneqq \{ 1, 4, 5, 7, 8, 9, 10, 11, 12, 14, 15, 16, 17, 18, 19, 20, 21, 22, 23, 24, 25, 26, 27, 28 \} \\
	L_{3}  & \coloneqq \{ 2, 3, 6, 7, 8, 9, 10, 11, 12, 13, 14, 15, 16, 17, 18, 19, 20, 21, 22, 23, 24, 25, 26, 27 \} \\
	L_{4}  & \coloneqq \{ 3, 4, 5, 6, 7, 8, 9, 12, 13, 15, 16, 17, 18, 19, 20, 21, 22, 23, 24, 25, 26, 27, 28 \} \\
	L_{5}  & \coloneqq \{ 2, 4, 5, 6, 7, 9, 10, 12, 13, 14, 16, 17, 18, 19, 20, 21, 22, 23, 24, 25, 26, 27, 28 \} \\
	L_{6}  & \coloneqq \{ 2, 3, 5, 6, 7, 10, 11, 13, 14, 15, 16, 17, 18, 19, 20, 21, 22, 23, 24, 25, 26, 27, 28 \} \\
	L_{7}  & \coloneqq \{ 3, 4, 5, 6, 8, 9, 10, 11, 12, 13, 14, 15, 16, 19, 20, 21, 22, 23, 24, 25, 26, 27, 28 \} \\
	L_{8}  & \coloneqq \{ 2, 3, 4, 7, 8, 9, 10, 11, 12, 13, 14, 15, 17, 19, 20, 21, 22, 23, 24, 25, 26, 27, 28 \} \\
	L_{9}  & \coloneqq \{ 1, 3, 5, 7, 8, 9, 10, 11, 12, 13, 14, 16, 17, 18, 19, 20, 22, 23, 24, 25, 26, 27, 28 \} \\
	L_{10} & \coloneqq \{ 2, 4, 5, 6, 7, 8, 11, 13, 14, 15, 16, 17, 18, 19, 20, 21, 22, 23, 24, 25, 26, 27, 28 \} \\
	L_{11} & \coloneqq \{ 1, 2, 6, 7, 8, 9, 10, 11, 12, 13, 14, 15, 16, 17, 18, 19, 21, 23, 24, 25, 26, 27, 28 \} \\
	L_{12} & \coloneqq \{ 1, 4, 5, 6, 10, 11, 12, 13, 14, 15, 16, 17, 18, 19, 20, 21, 22, 23, 24, 25, 26, 27, 28 \} \\
	L_{13} & \coloneqq \{ 2, 4, 5, 7, 8, 9, 10, 11, 12, 13, 14, 15, 16, 17, 18, 19, 20, 21, 22, 23, 24, 25, 27, 28 \} \\
	L_{14} & \coloneqq \{ 1, 4, 6, 7, 8, 9, 10, 11, 12, 13, 14, 15, 16, 17, 18, 20, 21, 22, 23, 24, 25, 26, 27, 28 \} \\
	L_{15} & \coloneqq \{ 1, 2, 3, 4, 5, 6, 7, 8, 9, 10, 11, 12, 13, 14, 15 \} 
\end{align*}
\endgroup
Next, we construct non-separable subsets with respect to $ \winningeu $ that consist of the above losing coalitions.
In order to verify that these subsets are indeed non-separable, the following lemma is helpful.

\begin{table}
	\begin{center}
		\footnotesize
		\begin{tabular}{rlrrlr}
			 \hline
			 \# & Member state & Population & \# & Member state & Population \\
			 \hline
			 1 & Germany        & 80 780 000 & $ \qquad $ 15 & Austria    & 8 507 786 \\
			 2 & France         & 65 856 609 &            16 & Bulgaria   & 7 245 677 \\
			 3 & United Kingdom & 64 308 261 &            17 & Denmark    & 5 627 235 \\
			 4 & Italy          & 60 782 668 &            18 & Finland    & 5 451 270 \\
			 5 & Spain          & 46 507 760 &            19 & Slovakia   & 5 415 949 \\
			 6 & Poland         & 38 495 659 &            20 & Ireland    & 4 604 029 \\
			 7 & Romania        & 19 942 642 &            21 & Croatia    & 4 246 700 \\
			 8 & Netherlands    & 16 829 289 &            22 & Lithuania  & 2 943 472 \\
			 9 & Belgium        & 11 203 992 &            23 & Slovenia   & 2 061 085 \\
			10 & Greece         & 10 992 589 &            24 & Latvia     & 2 001 468 \\
			11 & Czech Republic & 10 512 419 &            25 & Estonia    & 1 315 819 \\
			12 & Portugal       & 10 427 301 &            26 & Cyprus     &   858 000 \\
			13 & Hungary        &  9 879 000 &            27 & Luxembourg &   549 680 \\
			14 & Sweden         &  9 644 864 &            28 & Malta      &   425 384 \\
			 \hline
		\end{tabular}
	\end{center}
	\caption{Population data of the European Union on 01.01.2014, see also \cite[Table~1]{Kurz_2015}.}
	\label{tabpopulation}
\end{table}

\begin{lemma}\label{lemma:non-separable}
	Let $ \winning $ be a simple game and let $ \winning^* $ and $ \nonsep $ be sets of winning and losing coalitions for $ \winning $, respectively, such that $ |\winning^*| \ge |\nonsep| $.
	If
	\[
		| \{ W \in \winning^* : m \in W \} | = | \{ L \in \nonsep : m \in L \} |
	\]
	holds for all $ m \in M $, then $ \nonsep $ is non-separable with respect to $ \winning $.
\end{lemma}
\begin{proof}
	Consider any weighted game $ \winning' = \{ C \subseteq M : \sum_{m\in C} a_m \geq \beta \} \supseteq \winning $ with $ a \in \R_{\ge 0}^M $ and $ \beta \in \R $.
	Then we have
	\[
		\sum_{L \in \nonsep} \sum_{m \in L} a_m = \sum_{W \in \winning^*} \sum_{m\in W} a_m \ge \beta \left \vert \winning^* \right \vert.
	\]
	The last inequality holds because all elements of $ \winning^* $ are contained in $ \winning' $.
	Thus, there must exist some $ L \in \nonsep $, such that 
	\[
		\sum_{m \in L} a_m \ge \beta \cdot \frac{|\winning^*|}{|\nonsep|} \ge \beta.
	\]
	Therefore, we have $ L \in \winning' $ and hence $ \winning' \cap \nonsep \ne \emptyset $.
	Since this holds for any weighted game $ \winning' \supseteq \winning $, $ \nonsep $ is non-separable with respect to $ \winning $.
\end{proof}

We claim that the following $ 2 $-element subsets of the above losing coalitions are non-separable.
\begingroup
\allowdisplaybreaks
\begin{align}
	\nonumber & \{ L_{1}, L_{5}  \}, \{ L_{1}, L_{8}  \}, \{ L_{1}, L_{9}  \}, \{ L_{1},L_{10} \}, \{ L_{1},L_{11} \}, \{ L_{1},L_{13} \}, \{ L_{1},L_{14} \}, \{ L_{1},L_{15} \}, \\
	\nonumber & \{ L_{2}, L_{3}  \}, \{ L_{2}, L_{4}  \}, \{ L_{2}, L_{5}  \}, \{ L_{2},L_{6}  \}, \{ L_{2},L_{7}  \}, \{ L_{2},L_{8}  \}, \{ L_{2},L_{10} \}, \{ L_{2},L_{11} \}, \{ L_{2},L_{15} \}, \\
	\nonumber & \{ L_{3}, L_{4}  \}, \{ L_{3}, L_{5}  \}, \{ L_{3}, L_{7}  \}, \{ L_{3},L_{9}  \}, \{ L_{3},L_{10} \}, \{ L_{3},L_{12} \}, \{ L_{3},L_{13} \}, \{ L_{3},L_{14} \}, \{ L_{3},L_{15} \}, \\
	\nonumber & \{ L_{4}, L_{6}  \}, \{ L_{4}, L_{8}  \}, \{ L_{4}, L_{9}  \}, \{ L_{4},L_{11} \}, \{ L_{4},L_{12} \}, \{ L_{4},L_{13} \}, \{ L_{4},L_{14} \}, \{ L_{4},L_{15} \}, \\
	\nonumber & \{ L_{5}, L_{7}  \}, \{ L_{5}, L_{8}  \}, \{ L_{5}, L_{11} \}, \{ L_{5},L_{14} \}, \{ L_{5},L_{15} \}, \\
	\nonumber & \{ L_{6}, L_{7}  \}, \{ L_{6}, L_{8}  \}, \{ L_{6}, L_{9}  \}, \{ L_{6},L_{11} \}, \{ L_{6},L_{13} \}, \{ L_{6},L_{14} \}, \{ L_{6},L_{15} \}, \\
	\nonumber & \{ L_{7}, L_{9}  \}, \{ L_{7}, L_{10} \}, \{ L_{7}, L_{11} \}, \{ L_{7},L_{13} \}, \{ L_{7},L_{14} \}, \{ L_{7},L_{15} \}, \\
	\nonumber & \{ L_{8}, L_{9}  \}, \{ L_{8}, L_{10} \}, \{ L_{8}, L_{11} \}, \{ L_{8},L_{12} \}, \{ L_{8},L_{14} \}, \{ L_{8},L_{15} \}, \\
	\nonumber & \{ L_{9}, L_{10} \}, \{ L_{9}, L_{11} \}, \{ L_{9}, L_{12} \}, \{ L_{9},L_{13} \}, \{ L_{9},L_{14} \}, \{ L_{9},L_{15} \}, \\
	\nonumber & \{ L_{10},L_{11} \}, \{ L_{10},L_{14} \}, \{ L_{10},L_{15} \}, \\
	\nonumber & \{ L_{11},L_{12} \}, \{ L_{11},L_{13} \}, \{ L_{11},L_{15} \}, \\
	\nonumber & \{ L_{12},L_{13} \}, \{ L_{12},L_{15} \}, \\
	\nonumber & \{ L_{13},L_{14} \}, \{ L_{13},L_{15} \}, \\
	\label{eq2edges} & \{ L_{14},L_{15} \} 
\end{align}
\endgroup

To see that each above set $ \nonsep \coloneqq \{L_i, L_j\} $ is non-separable, we make use of \cref{lemma:non-separable} as follows.
If $ L_i,L_j \ne L_{15} $, we have that $ L_i $ and $ L_j $ are contained in $ \winning_{\ref{enum:55}} \setminus \winning_{\ref{enum:65}} $.
Pick a set of states $ A \subseteq L_i \cup L_j \setminus (L_i \cap L_j) $ of minimum total population such that $ W_1 := A \cup (L_i \cap L_j) $ is contained in $ \winning_{\ref{enum:25}} \subseteq \winning $.
For all above pairs it can be checked that $ W_2 := (L_i \cup L_j) \setminus A $ is contained in $ \winning_{\ref{enum:55}} \cap \winning_{\ref{enum:65}} \subseteq \winning $.
By construction, $ \nonsep $ and $ \winning^* \coloneqq \{W_1, W_2\} $ satisfy the assumptions of \cref{lemma:non-separable} and hence $ \nonsep $ is indeed non-separable.

Otherwise, we may assume that $ L_j = L_{15} $.
For all above pairs, exchanging the two members with the least population in $ L_i \setminus L_{15} $ with the member of largest population in $ L_{15} \setminus L_i $, results in two winning sets $ W_1, W_2 $.
Again, $ \nonsep $ and $ \winning^* \coloneqq \{W_1, W_2\} $ satisfy the assumptions of \cref{lemma:non-separable}, implying that $ \nonsep $ is non-separable.

Moreover, the following $ 3 $-element subsets of losing coalitions are also non-separable.
\begin{equation}
	\label{eq3edges}
	\{ L_{1},L_{2}, L_{12} \},
	\{ L_{1},L_{4}, L_{7}  \},
	\{ L_{1},L_{6}, L_{12} \},
	\{ L_{4},L_{5}, L_{10} \},
	\{ L_{5},L_{10},L_{12} \}
\end{equation}
To see that these sets are non-separable, consider the following sets of winning coalitions.
\begingroup
\allowdisplaybreaks
\begin{align*}
	W_{1}  & \coloneqq \{ 1, 2, 4, 5, 6, 13, 14, 16, 17, 18, 19, 20, 21, 22, 23, 24, 25, 26, 27, 28 \} \\
	W_{2}  & \coloneqq \{ 4, 5, 6, 7, 8, 9, 10, 11, 12, 13, 14, 15, 16, 17, 18, 19, 20, 21, 22, 23, 24, 25, 26, 27, 28 \} \\
	W_{3}  & \coloneqq \{ 3, 4, 5, 6, 8, 9, 10, 11, 12, 13, 15, 16, 17, 18, 19, 20, 21, 22, 23, 24, 25, 26, 27, 28 \} \\
	W_{4}  & \coloneqq \{ 1, 2, 5, 6, 10, 11, 12, 13, 14, 15, 16, 17, 18, 19, 20, 21, 22, 23, 24, 25, 26, 27, 28 \} \\
	W_{5}  & \coloneqq \{ 2, 3, 4, 5, 6, 7, 13, 16, 17, 18, 19, 20, 21, 22, 23, 24, 25, 26, 27, 28 \} \\
	W_{6}  & \coloneqq \{ 2, 4, 5, 6, 7, 10, 11, 12, 13, 14, 15, 16, 17, 18, 19, 20, 21, 22, 23, 24, 25, 26, 27, 28 \} \\
	W_{7}  & \coloneqq \{ 1, 3, 4, 5, 9, 10, 11, 12, 14, 15, 16, 17, 18, 19, 20, 21, 22, 23, 24, 25, 26, 27, 28 \} \\
	W_{8}  & \coloneqq \{ 2, 3, 4, 5, 6, 10, 11, 15, 16, 17, 18, 19, 20, 21, 22, 23, 24, 25, 26, 27, 28 \} \\
	W_{9}  & \coloneqq \{ 2, 4, 5, 6, 7, 8, 9, 12, 13, 14, 15, 16, 17, 18, 19, 20, 21, 22, 23, 24, 25, 26, 27, 28 \} \\
	W_{10} & \coloneqq \{ 3, 5, 6, 7, 8, 9, 10, 11, 12, 13, 14, 15, 16, 17, 18, 19, 20, 21, 22, 23, 24, 25, 26, 27, 28 \} \\
	W_{11} & \coloneqq \{ 1, 2, 5, 6, 8, 10, 11, 12, 15, 16, 17, 18, 19, 20, 21, 22, 23, 24, 25, 26, 27, 28 \} \\
	W_{12} & \coloneqq \{ 2, 3, 4, 5, 6, 8, 9, 12, 15, 16, 19, 20, 21, 22, 23, 24, 25, 26, 27, 28 \}
\end{align*}
\endgroup
Observing that the pairs
\begin{align*}
	& (\{ W_{2}, W_{7} , W_{11} \}, \{ L_{1}, L_{2}, L_{12} \}), \\
	& (\{ W_{3}, W_{10}, W_{12} \}, \{ L_{1}, L_{4}, L_{7} \}), \\
	& (\{ W_{4}, W_{8} , W_{10} \}, \{ L_{1}, L_{6}, L_{12} \}), \\
	& (\{ W_{2}, W_{5} , W_{9}  \}, \{ L_{4}, L_{5}, L_{10} \}), \text{ and} \\
	& (\{ W_{1}, W_{2} , W_{6}  \}, \{ L_{5}, L_{10}, L_{12} \})
\end{align*}
satisfy the assumptions of \cref{lemma:non-separable}, we see that the sets in~\eqref{eq3edges} are indeed non-separable.

In the next section, we show that the non-separable sets in~\eqref{eq2edges} and~\eqref{eq3edges} do not admit a $ 7 $-cover.
Recall that this implies that the dimension must be at least $ 8 $ by \cref{lemma:dim}.

\section{Proof that no 7-cover can exist}
\label{secLowerBound}
For the sake of contradiction, let us assume that the non-separable sets in~\eqref{eq2edges} and~\eqref{eq3edges} admit a $ 7 $-cover.
This implies that there exist sets $ \losing_1,\dots,\losing_7 \subseteq \{L_1,\dots,L_{15}\} $ such that
\begin{itemize}
\item[(i)] each $ \losing_j $ is an inclusion-wise maximal subset of $ \{L_1,\dots,L_{15}\} $ that does not contain any of the sets in~\eqref{eq2edges} and~\eqref{eq3edges}, and
\item[(ii)] $ \losing_1 \cup \dots \cup \losing_7 = \{L_1,\dots,L_{15}\} $.
\end{itemize}
It can be easily verified that the only sets satisfying~(i) are the following.
\begin{align}
	\nonumber &\{ L_{1},L_{2} \}, \{ L_{1},L_{3} ,L_{6} \},\{ L_{1},L_{4}        \}, \{ L_{1},L_{7},L_{12} \}, \{ L_{2},L_{9}  \}, \{ L_{2}, L_{12},L_{14} \},\{ L_{2},L_{13}        \},\\
	\nonumber &\{ L_{3},L_{8} \}, \{ L_{3},L_{11}       \},\{ L_{4},L_{5}        \}, \{ L_{4},L_{7}        \}, \{ L_{4},L_{10} \}, \{ L_{5}, L_{6} ,L_{10} \},\{ L_{5},L_{6} ,L_{12} \},\\
	\label{choices} &\{ L_{5},L_{9} \}, \{ L_{5},L_{10},L_{13}\},\{ L_{6},L_{10},L_{12}\}, \{ L_{7},L_{8}        \}, \{ L_{8},L_{13} \}, \{ L_{11},L_{14}        \},\{ L_{15}              \}
\end{align}
In what follows, for a weight-vector $ w = (w_1,\dots,w_{15}) \in \R^{15} $, let us define the \emph{weight} of a set $ \losing' \subseteq \{1,\dots,15\} $ as $ w(\losing') := \sum_{i \in \losing'} w_i $.

Suppose first that none of the sets $ \losing_1,\dots,\losing_7 $ is equal to $ \{ L_{1},L_{3} ,L_{6} \} $.
In this case, consider the weight-vector
\[
	w = (\nicefrac{1}{2},0,1,\nicefrac{1}{2},0,1,\nicefrac{1}{2},0,1,0,0,0,1,1,1)
\]
and observe that the weight of each set in~\eqref{choices} that is distinct from $ \{ L_{1},L_{3} ,L_{6} \} $ is at most $ 1 $.
Thus, the weight of each set $ \losing_1,\dots,\losing_7 $ is at most $ 1 $, and we obtain
\[
	7 < \tfrac{15}{2} = w(\{1,\dots,15\}) = w(\losing_1 \cup \dots \cup \losing_7) \le w(\losing_1) + \dots + w(\losing_7) \le 7,
\]
a contradiction.

It remains to consider the case that one of the sets $ \losing_1,\dots,\losing_7 $ is equal to $ \{ L_{1},L_{3} ,L_{6} \} $, say $ \losing_1 $.
Consider the weight-vector
\[
	w = (0,\nicefrac{1}{3},0,\nicefrac{2}{3},\nicefrac{1}{3},0,\nicefrac{1}{3},\nicefrac{2}{3},\nicefrac{2}{3},\nicefrac{1}{3},1,\nicefrac{2}{3},\nicefrac{1}{3},0,1)
\]
and observe that the weight of each set in~\eqref{choices} is at most $ 1 $, and that $ w(\losing_1) = 0 $.
Thus, we have
\[
	6 < \tfrac{19}{3} = w(\{1,\dots,15\}) = w(\losing_1 \cup \dots \cup \losing_7) \le w(\losing_2) + \dots + w(\losing_7) \le 6,
\]
another contradiction.
This completes our proof.

\section*{Acknowledgements}
The second author would like to thank Gerhard Woeginger for bringing this topic to his attention.

\bibliographystyle{plainnat}
\bibliography{references.bib}

\end{document}